\documentclass[12pt, reqno]{amsart}

\usepackage{amssymb}
\usepackage{amsmath}
\usepackage{amscd}
\usepackage{cancel}
\usepackage{float}
\usepackage{graphicx}
\usepackage{amsfonts}
\usepackage{color}
\usepackage{xypic}

\usepackage{amsfonts}
\usepackage[cmtip,arrow]{xy}
\usepackage{cancel}
\usepackage[normalem]{ulem}
\newtheorem{theorem}{Theorem}
\newtheorem{corollary}{Corollary}
\newtheorem{lemma}{Lemma}
\newtheorem{proposition}{Proposition}
\newtheorem{definition}{Definition}
\newtheorem{remark}{Remark}
\newtheorem{example}{Example}

\newcommand{\CC}{\mathbb{C}}

\newcommand{\KK}{\mathbb{K}}

\newcommand{\NN}{\mathbb{N}}
\newcommand{\RR}{\mathbb{R}}

\newcommand{\OO}{\widetilde{\Omega}}
\newcommand{\TL}{L}

\def\a{\mathsf{a}}
\def\b{\mathsf{b}}
\def\c{\mathsf{c}}

\def\u{\mathsf{u}}
\def\v{\mathsf{v}}
\def\w{\mathsf{w}}

\def\z{\mathsf{z}}

\def\En{\mathcal{E}_n}

\def\P{\mathrm{P}}

\newcommand{\E}{\mathcal{E}}
\newcommand{\LL}{\mathbb{L}}

\newcommand{\LLL}{\mathfrak{L}}

\def\a{\mathsf{a}}
\def\b{\mathsf{b}}
\def\c{\mathsf{c}}

\def\u{\mathsf{u}}
\def\v{\mathsf{v}}
\def\w{\mathsf{w}}

\def\1{{\bf 1}}

\begin{document}

\title{Two parameters  bt--algebra and   invariants for links and tied links}

\author{F. Aicardi}
 \address{ Sistiana Mare 56,  34011  Trieste, Italy.}
\email{francescaicardi22@gmail.com}
 \author{J. Juyumaya}
 \address{IMUV,
 Gran Breta\~{n}a 1111, Valpara\'{i}so 2340000, Chile.}
 \email{juyumaya@gmail.com}

\thanks{The second  author was  in part supported   by the Fondecyt Project  1180036}


\begin{abstract}
We introduce a two--parameters bt--algebra which,  by specialization,   becomes
the one--parameter bt--algebra, introduced by the authors, as well as
another one--parameter presentation of it; the invariant    for  links and  tied links,
 associated to this two--parameter algebra via Jones recipe,
contains  as specializations the invariants obtained from these two
presentations of the bt--algebra and then   is more powerful than each of
them.  Also, a new non Homflypt   polynomial
invariant   is obtained for links, which is related to  the linking matrix.
\end{abstract}

\maketitle

\section{Introduction}
The bt--algebra is    a  one--parameter finite dimensional  algebra defined by generators and relations, see \cite{ajICTP1, rhJAC}.
In  \cite{maIMRN}  it is  shown  how   to  associate to  each  Coxeter group  a certain algebra, and in  the case of   the Weyl group of type $A$ this  algebra  coincides
with the bt--algebra;  this may open new perspectives for the  study  of  the bt--algebra  in knot theory,   cf. \cite{flJPAA}.  The representation theory of the bt--algebra has been studied in \cite{rhJAC, esryJPAA, jaPoJLMS}.
 \smallbreak
For every positive integer $n$, we denote by $\En (\u )$ the bt--algebra  over $\CC (\u)$, with parameter $\u$. The original definition of $\En (\u )$ is by braid generators $T_1, \ldots , T_{n-1}$ and tie  generators $E_1,\ldots ,E_{n-1}$, satisfying the defining generators of the tied braid monoid defined in \cite[Definition 3.1 ]{aijuJKTR1} together with the polynomial relation
$$
T_i^2 = 1 + (\u -1) E_i + (\u-1) E_iT_i, \quad \text{for all $i$}.
$$

 It is known that the bt--algebra is a knot algebra: indeed,  in \cite{aijuMMJ1} we have defined a three--variable polynomial invariant for classical links which is denoted by $\bar\Delta$; this invariant   was constructed  originally  by using the method - also called  Jones recipe - that Jones introduced  to construct the Homflypt polynomial \cite{joAM}.
\smallbreak
In \cite{chjukalaIMRN}    another presentation for the bt--algebra is considered. More precisely,  denote by $\sqrt{\u}$ a variable such that $(\sqrt{\u})^2 =\u $: the new presentation of the bt--algebra is now over $\CC(\sqrt{\u})$ and is presented by the same  tie  generators $E_i$'s  but the generators $T_i$'s are replaced  by braid generators $V_i$'s,  still
 satisfying all original defining relation  of  the    $T_i$'s   with exception of the polynomial relation above which is replaced  by
$$
V_i^2 = 1 + (\sqrt{\u}- \sqrt{\u}^{\,-1}) E_iV_i, \quad \text{for all $i$}.
$$
We denote by $\E (\sqrt{\u})$ the bt--algebra with this new presentation. Now, again, by using the Jones recipe on the bt--algebra but with the presentation $\E (\sqrt{\u})$,  a  three--variable polynomial invariant for classical links is constructed in \cite{chjukalaIMRN}; this invariant is denoted by $\Theta$.
\smallbreak
It was noted by the first author that $\bar\Delta$ and $\Theta$ are not topologically equivalent, see \cite{aica}, cf. \cite{chjukalaIMRN}; this is an amazing fact that shows the   subtlety  of the Jones recipe. In fact,
the main motivation of this note is to understand the relation between the invariants $\bar\Delta$ and $\Theta$. To do that we introduce a bt--algebra  with two commuting parameters $\u$ and $\v$, denoted by $\En(\u, \v)$, presented by the tie   generators $E_i$'s and braid generators $R_i$'s, subject to the same monomial relations as the bt--algebra and the polynomial relations
$$
R_i^2 = 1 + (\u -1) E_i + (\v-1) E_iR_i, \quad \text{for all $i$}.
$$
  Similarly to what happens for  the two--parameters   Hecke algebra   \cite[Subsection 4.2]{katu}, the bt--algebra with two parameters is isomorphic to the bt--algebra with one parameter, see Proposition \ref{Isomorphism}; this fact allows to define a Markov trace on $\En(\u, \v)$ (Proposition \ref{MarkovTrace}).  Consequently, we apply the Jones recipe to the bt--algebra with two parameters, obtaining  a   four--variable invariant  polynomial,  denoted by $\Upsilon$, for classical links as well its  extension $\widetilde{\Upsilon}$ to tied links \cite{aijuJKTR1}.  As  it  will be  observed  in Remark \ref{Specializations},  specializations of the parameters in   $\En(\u,\v)$ yields $\En(\u)$ and  $\En(\sqrt{\u})$; therefore,  the respective
 specializations  of $\Upsilon$ yields  the invariants $\bar\Delta$ and $\Theta$; this gives an answer to the initial question that
motivated this work.
\smallbreak
  In   Section \ref{TiedUpsilon} we define $\widetilde{\Upsilon}$ by skein relations.  We also give a close look  to the specialization of
$\widetilde{\Upsilon}$ at $\v=1$, which is denoted by
$\widetilde{\Omega}$. In Theorem \ref{T4}  we     show some  properties
of $\widetilde{\Omega}$     by introducing   a generalization of the linking number to
tied links.   Finally, in Section \ref{Computations} we give a table comparing the invariant $\Upsilon$ and its specializations considered here.
  Section \ref{Digression} is a digression on the bt--algebra, at one and two parameters, in comparison  with two different presentations of the  Hecke algebra.

 \section{Preliminaries}
Here,   $\KK$--algebra  means   an associative algebra, with unity $1$, over the field $\KK$.

\subsection{}
 As usual we denote by $B_n$ the braid group on $n$--strands. The Artin presentation of $B_n$ is by the braids  generators $\sigma_1,\ldots ,\sigma_{n-1}$ and the relations: $\sigma_i\sigma_j=\sigma_j\sigma_i$, for $\vert i-j\vert >1$ and $\sigma_i\sigma_j\sigma_i = \sigma_j\sigma_i\sigma_j$, for $\vert i-j\vert =1$.
An extension of the braid group $B_n$ is the {\it monoid of tied braids} $TB_n$, which is a master piece in   the study  of tied links.

\begin{definition}\cite[Definition 3.1]{aijuJKTR1}
$T\!B_n$ is  the monoid presented by  the usual  braids generators
 $\sigma_1, \ldots , \sigma_{n-1}$  together with  the tied generators   $\eta_1, \ldots ,\eta_{n-1}$ and  the relations:
\begin{eqnarray}
\label{eta1}
\eta_i\eta_j & =  &  \eta_j \eta_i \qquad \text{ for all $i,j$},\\
\label{eta2}
 \eta_i\eta_i & = & \eta_i \qquad \text{ for all $i$},\\
 \label{eta3}
\eta_i\sigma_i  & = &   \sigma_i \eta_i \qquad \text{ for all $i$},\\
\label{eta4}
\eta_i\sigma_j  & = &   \sigma_j \eta_i \qquad \text{for $\vert i  -  j\vert >1$},\\
\label{eta5}
\eta_i \sigma_j \sigma_i & = &   \sigma_j \sigma_i \eta_j \qquad\text{ for $\vert i - j\vert =1$},\\
\label{eta6}
\eta_i\eta_j\sigma_i & = & \eta_j\sigma_i\eta_j  \quad = \quad\sigma_i\eta_i\eta_j \qquad \text{ for $\vert i  -  j\vert =1$},\\
\label{eta7}
\sigma_i\sigma_j  & = &   \sigma_j \sigma_i \qquad \text{for $\vert i  -  j\vert >1$},\\
\label{eta8}
\sigma_i \sigma_j \sigma_i & = &   \sigma_j \sigma_i \sigma_j \qquad\text{ for $\vert i - j\vert =1$},\\
\label{eta9}
\eta_i\sigma_j\sigma_i^{-1} & = &  \sigma_j\sigma_i^{-1}\eta_j \qquad \text{ for $\vert i  -  j\vert =1$}.
\end{eqnarray}
\end{definition}

\subsection{}
Set $\u$ a variable: the bt--algebra $\En(\u)$  \cite{ajICTP1, rhJAC, aijuMMJ1} can be conceived as the  quotient algebra of the monoid algebra of $TB_n$  over $\CC(\u)$, by the two--sided ideal generated by
$$
\sigma_i^2 - 1-(\u-1)\eta_i(1+\sigma_i),\quad \text{for all}\quad i.
$$
See \cite[Remark 4.3]{aijuJKTR1}. In other words, $\En(\u)$  is the $\CC(\u)$--algebra generated by $T_1,\ldots ,T_{n-1}$, $E_1,\ldots ,E_{n-1}$ satisfying the relations (\ref{eta1})--(\ref{eta8}), where $\sigma_i$ is replaced by $T_i$ and $\eta_i$ by $E_i$, together with the relations
\begin{equation}\label{FirstQuadratic}
T_i^2 = 1 + (\u-1)E_i + (\u-1)E_iT_i, \quad \text{for all}\quad i.
\end{equation}
We  consider now another  presentation of the bt--algebra,    used in \cite{chjukalaIMRN, maIMRN}.   Let  $\sqrt{\u}$ be a variable s.t. $\sqrt{\u}^2=\u$. We denote by $\En(\sqrt{\u})$ the bt--algebra presented by the generators  $V_1,\ldots , V_{n-1}$ and $E_1, \ldots, E_{n-1}$, where
$$
V_i:= T_i  + \left( \frac{1}{\sqrt{\u}}-1 \right) E_iT_i.
$$
 The $V_i$'s still satisfy the defining relations   (\ref{eta1}) to (\ref{eta8}), substituting $\sigma_i$ with $V_i$, $\eta_i$ with $E_i$, but equation (\ref{FirstQuadratic}) becomes
\begin{equation}\label{SecondQuadratic}
V_i^2 = 1 +   \left(\sqrt{\u} -\frac{1}{\sqrt{\u}}\right) E_i V_i, \quad \text{for all}\quad i.
\end{equation}

\subsection{}
  Tied links were introduced in \cite{aijuJKTR1} and roughly correspond   to  links  which may have   ties connecting pairs of  points  of  two  components or of the  same  component.      The   ties   in the diagrams  of the tied links  are   drawn  as  springs,  to outline   the fact  that they  can be contracted  and  extended,  letting  their  extremes  to  slide  along the  components.   The  ties  define a partition of  the set of components in this  way: two  components  connected by a  tie belong to  the same  part of the  partition.
   Every  classical link can be considered as a  tied  link  without ties; in this  case  each  component form a  distinct part of  the partition. Alternatively,  a classical  link can  be considered as a  tied link  in which all components form a   sole part.

We denote by $ \mathfrak{L}$ the set of classical  links in  $\RR^3$ and by $\widetilde{ \mathfrak{L}}$ the set of tied links.  As we have just recalled,  $ \mathfrak{L}\subset \widetilde{ \mathfrak{L}}$,  but  the set $\mathfrak{L}$   can be identified  also with  the subset  $\widetilde{ \mathfrak{L}}^*$  of $\widetilde{ \mathfrak{L}}$, formed by the tied links whose components are all tied. In terms of braids, the situation is as follows. Recall that the tied links are in bijection with  the equivalence classes of $TB_{\infty}$ under the t--Markov moves \cite[Theorem 3.7]{aijuJKTR1}. Now, observe  that
    $B_n$ can be   naturally considered as a submonoid  of $TB_n$ and    the t--Markov moves at level of $B_n$ are  the classical Markov moves: this  implies the inclusion $ \mathfrak{L}\subset \widetilde{ \mathfrak{L}}$. On the other hand,  the group $B_n$  is isomorphic, as group, to the submonoid $EB_n$ of $TB_n$,
$$
EB_n: =\{\eta^n\sigma \,;\, \sigma\in B_n\},\qquad \eta^n:=\eta_1\cdots \eta_{n-1},
$$
where the group  isomorphism from $EB_n$ to $B_n$, denoted by $\mathtt{f}$,  is given by $\mathtt{f}(\eta^n\sigma)=\sigma$. Moreover,   two tied braids of $EB_n$  are t--Markov equivalent  if and only if their images by
$\mathtt{f}$ are Markov equivalent. This  explains, in terms of braids,   the identification between   $ \mathfrak{L}$ and $\widetilde{ \mathfrak{L}}^*$  mentioned above.   For  more details see \cite[Subsection 2.3]{aijuMathZ}.
\subsection{}\label{NotationFact}
Invariants for classical and tied links  were constructed by using the bt--algebra in the Jones recipe   \cite{joAM}. We   recall some  facts and   introduce    some notations  for  these invariants:
\begin{enumerate}
\item   $\Delta$ and $\widetilde{\Delta}$  denote respectively  the three--variable    invariant for classical links and tied links, defined through the original  bt--algebra.  The invariant $\Delta$, called $\bar \Delta$ in  \cite{aijuMMJ1},   is  the   restriction of $\widetilde{\Delta}$  to $   \mathfrak{L}$;   the invariant $ \widetilde{\Delta}$ was defined in \cite{aijuJKTR1}, where was denoted $\mathcal{F}$.
\item  $\Theta$ and $\widetilde{\Theta}$  denote respectively the three--variable   invariant for classical links and tied links, defined   in \cite{chjukalaIMRN}; the original notation for $\widetilde{\Theta}$ was $\overline{\Theta}$.    Notice that  the  invariant $\Theta$ is the  restriction of  $\widetilde{\Theta}$  to   $ \mathfrak{L}$.
\item  The    invariants  $\widetilde{\Delta}$ and   $\widetilde{\Theta}$, restricted   to $\widetilde{ \mathfrak{L}}^*$,   coincide  with  the Homplypt polynomial, which is denoted by $\P =\P(t,x)$; we keep the   defining skein  relation of $\P$ as in \cite[Proposition 6.2]{joAM}.
\item The invariants $\Delta$ and $\Theta$ coincide with the Homplypt polynomial,   whenever they are evaluated on knots; however they distinguish pairs of links that are not distinguished by $\P$.   See  \cite[Theorem 8.3]{chjukalaIMRN} and  \cite[Proposition 2]{aica}.
\item It is intriguing to note that despite the only difference in the construction  $\Delta$ and $\Theta$ is the presentation used for the bt--algebra,   these invariants are not topologically equivalent,   see \cite{aica, chjukalaIMRN}.

\end{enumerate}

\section{The two--parameters bt--algebra}

\subsection{}
 Let $\v$ be  a variable commuting  with  $\u$,  and  set ${\Bbb K}= {\Bbb C}(\u,\v)$.
\begin{definition} [See \cite{ajICTP1, rhJAC, aijuMMJ1}] The two--parameter bt--algebra, denoted by $\E_n(\u,\v)$, is defined by $\E_1(\u,\v):=\KK$  and, for $n>1$, as   the unital associative $\KK$--algebra, with unity $1$,   presented  by the braid generators $R_1,\ldots , R_{n-1}$ and the tie generators $E_1,\ldots , E_{n-1}$, subject to the following relations:
\begin{eqnarray}
\label{bt1}
E_iE_j & =  &  E_j E_i \qquad \text{ for all $i,j$},
\\
\label{bt2}
E_i^2 & = &  E_i\qquad \text{ for all $i$},
\\
\label{bt3}
E_i R_j  & = &
 R_j E_i \qquad \text{for $\vert i  -  j\vert >1$},
\\
\label{bt4}
E_i R_i & = &   R_i E_i  \qquad \text{ for all $i$},
\\
\label{bt5}
E_iR_jR_i & = &  R_jR_iE_j \qquad \text{ for $\vert i  -  j\vert =1$},
\\
\label{bt6}
E_iE_jR_i & = & E_j R_i E_j  \quad = \quad R_iE_iE_j \qquad \text{ for $\vert i  -  j\vert =1$},
\\
\label{bt7}
R_i R_j & = & R_j R_i\qquad \text{ for $\vert i - j\vert > 1$},
 \\
 \label{bt8}
R_i R_j R_i& = & R_j R_iR_j\qquad \text{ for $\vert i - j\vert = 1$},
 \\
 \label{bt9}
R_i^2 & = & 1  + (\u-1)E_i + (\v-1)E_i R_i\qquad \text{ for all $i$}.
\end{eqnarray}
\end{definition}

Notice that every  $R_i$ is invertible,  and
\begin{equation}\label{Tinverse}
R_i^{-1} = R_i + (1-\v)\u^{-1}E_i + (\u^{-1} - 1)E_iR_i.
\end{equation}

\begin{remark}\rm
  The algebra $\E_n(\u,\v)$ can be conceived as the quotient of the monoid algebra  of  $TB_n$, over $\KK$, by the two--sided  ideal  generated by all expressions  of the form $\sigma_i^2 - 1  - (\u-1)\eta_i - (\v-1)\eta_i \sigma_i$, for all $i$.
\end{remark}
\begin{remark}\label{Specializations}\rm
Observe that the original bt--algebra $\E_n(\u)$   is obtained as   $\E_n(\u,\u)$,    while  the presentation   $\E_n(\sqrt{\u})$   corresponds to  $\E_n(1,\v)$,   with  $\v=\sqrt{\u} - \sqrt{\u}\,^ {-1}+1$.
\end{remark}

\subsection{}
 We  show  here  that the new two--parameters algebra is  isomorphic to  the original
bt--algebra.

Let  $\delta$ be  a root of the quadratic polynomial

\begin{equation}\label{QuadraticDelta}
   \u (\z+1)^2-  (\v-1) (\z  +1) -1.
\end{equation}

 Define the elements  $T_i$'s by
\begin{equation}\label{TiRi}
T_i:= R_i + \delta E_iR_i,\quad \text{for all}\quad i.
\end{equation}
\begin{proposition}\label{Isomorphism}
The $\LL$--algebras $\En (\u ,\v)\otimes_{\KK} \LL$ and $\En( \u ( \delta+1)^2  )$,   are isomorphic through the mappings $R_i\mapsto T_i$, $E_i\mapsto E_i$,     where  $\LL$  is the smaller  field containing $\KK$  and   $\delta$.
 \end{proposition}

\begin{proof}

The $T_i$'s satisfy  the relations (\ref{bt1})--(\ref{bt8}) and  we  have,  using relation (\ref{bt9}),
$$
T_i^2 = R_i^2+(\delta^2+2\delta)E_iR_i^2=1+(\u(\delta +1)^2-1)E_i+(\v-1)(\delta+1)^2 E_iR_i.
$$
Now,  since
$$
R_i = T_i -\frac{\delta}{\delta +1}E_iT_i,
$$
we have  $E_iR_i=  (\delta + 1)^{-1} E_iT_i$, and  substituting we get
\begin{equation}\label{quadra}
T_i^2
  =  1 + ( \u (\delta+1)^2-1)  E_i + (\v-1) (\delta  +1) E_iT_i.
\end{equation}
  Therefore,   the  coefficients of $E_i$ and $E_iT_i$   are equal since $\delta$ is a root of the polynomial (\ref{QuadraticDelta}).
\end{proof}

 \begin{remark}\rm
  Notice that the roots of (\ref{QuadraticDelta}) are:
$\z_{\pm}=(\v-1-2\u \pm \sqrt{(\v-1)^2+ 4 \u})/2\u$, so
\begin{equation}\label{quadraticT}
T_i^2 = 1 + \frac{(\v-1)\left(\v-1\pm\sqrt{(\v-1)^2+4\u}\right)}{2\u} ( E_i +  E_iT_i).
\end{equation}\rm
 Thus,  for  $\v=\u$, we  have: $ \z_+=0$  and $\z_-= - \u^{-1}(\u+1) $
with  the corresponding quadratic  relations:
$$   T_i^2= 1+ (\u-1) (E_i+E_iT_i), \quad   T_i^2= 1+ \frac{1-\u}{\u} (E_i+E_iT_i).$$
\end{remark}

  The  first solution gives  trivially $\E_n(\u)$, while  the  second one gives   another presentation of  $\E_n(\u)$,  obtained  by  keeping  as  parameter  $\u^{-1}$; note that $\E_n(\u) =\E_n(\u^{-1})$.

 On the other hand, for $\u=1$, we  get
$  \z_{\pm}=   ( \v-3 \pm \sqrt{\v^2-2 \v + 5  })/2  $
giving
\begin{equation}
T_i^2 = 1 + \frac{(\v-1)\left(\v-1\pm\sqrt{\v^2-2 \v + 5} \ \right)  }{2} ( E_i +  E_iT_i).
\end{equation}\rm

These  two solutions determine    isomorphisms between  $\E_n(\sqrt{\u})$ and  $\E_n(\u)$.

   At this point  we have to  note  that  there is  another interesting  specialization of  $\E_n(\u,\v)$,  namely  when $\v=1$.   In fact,  $\E_n(\u,1)$   deserves  a deeper investigation.  Here  we  gives some relations  holding only in  this specialization.
  More precisely, we have:
\begin{equation}\label{QuafraticInverse}
R_i^2 = 1 + (\u -1 )E_i\quad \text{and} \quad R_i^{-1} = R_i + (\u^{-1}-1)E_iR_i\quad \text{for all $i$.}
\end{equation}
Then we deduce
\begin{equation}\label{Cubic}
(\u +1) R_i -\u R_i^{-1} = R_i^3
\end{equation}
since $E_iR_i^{-1}=\u^{-1}E_iR_i$.
So,
\begin{equation}
R_i^4 -(\u +1 )R_i^2 + \u =0, \quad\text{or equivalently}\quad (R_i^2-1)( R_i^2 -\u )=0.
\end{equation}

\subsection{  Markov trace on  $\E_n(\u,\v)$}

\begin{proposition}\label{MarkovTrace}
Let $\a$ and $\b$ two mutually commuting variables.
There exists a unique Markov trace $\rho =\{\rho_n\}_{n\in\NN}$ on $\En (\u ,\v)$, where the $\rho_n$'s are linear maps from
 $\En (\u ,\v)$ to   $\LL(\a, \b)$, satisfying $\rho_n(1)=1$, and defined inductively by  the  rules:
\begin{enumerate}
\item $\rho_n (XY)=\rho_n(YX)$,
\item $\rho_{n+1}(XR_n)= \rho_{n+1}(XR_nE_n)=\a \rho_n(X) $,
\item $\rho_{n+1}(XE_n)= \b\rho_n(X)$,
\end{enumerate}
where  $X,Y\in \En(\u, \v)$.
\end{proposition}
\begin{proof}
 The proof follows  from   Proposition \ref{Isomorphism},   since is obtained   by carrying the    Markov trace on the bt--algebra  \cite[Theorem 3]{aijuMMJ1} to $\En (\u ,\v)$.   More precisely, if we denote by $\rho'$ the Markov trace on the bt--algebra, then $\rho$ is defined by $\rho'\circ \phi$, where $\phi$ denote the isomorphism of Proposition  \ref{Isomorphism}; moreover denoting by $\a'$ and $\b'$ the parameters trace of $\rho'$, we have  $\a= (\delta + 1)^{-1}\a'$ and $\b =\b'$.
\end{proof}

\section{Invariants}
  In this  section  we define,  via  Jones recipe, the  invariants of classical and tied links associated to  the  algebra $\E_n(\u,\v)$.
\subsection{}  Define the  homomorphism  $\pi_{\c}$  from  $B_n$ to $\E_n(\u,\v)$    by taking
\begin{equation}\label{homo}
\pi_{\c} (\sigma_i) = \sqrt{\c} R_i,
\end{equation}
where  the scaling factor $\c$ is obtained by imposing, due to the second Markov move,  that  $(\rho \circ \pi_{\c})(\sigma_i)=
(\rho \circ \pi_{\c})(\sigma_i^{-1})$; thus
\begin{equation}\label{formulac}
\c:= \frac{\a +(1-\v)\u^{-1}\b + (\u^{-1}-1)\a}{\a} =  \frac{\a+\b(1-\v)}{\a\u}.
\end{equation}
\begin{theorem} The  function
$\Upsilon :\mathfrak{L} \longrightarrow   \CC(\u,\v,\a,\sqrt{\c})$,  defines an invariant for classical links,
$$
\Upsilon (L) :=\left(\frac{1}{\a\sqrt{\c}}\right)^{n-1}(\rho \circ \pi_{\c})(\sigma),
$$
where $L = \widehat{\sigma}$,  $\sigma\in B_n$.
\end{theorem}

\begin{proof}
The  proof  follows step by step  the proof  done for the invariant  $\bar \Delta$ in \cite{aijuMMJ1},  replacing  the elements $T_i$ by  $R_i$.  Observe  that the only  differences consist  in  the expressions of  $\textsf{L}$  (see (46)\cite{aijuMMJ1}), that must be replaced by   $\c$,  and  of the  inverse element, that contains   now two parameters. However, it is  a routine to check that the proof is not  affected by  the  presence  of two  parameters instead of  one.
\end{proof}

\begin{remark}\rm
From Remark \ref{Specializations} it follows that, respectively,  the invariants $\Delta$ and $\Theta$ correspond  to the specializations  $\u=\v$ and $\u=1$ with $\v= \sqrt{\u} -\sqrt{\u}^{\,-1} $ of $\Upsilon$.
\end{remark}

\subsection{}

  The invariant $\Upsilon$ can be extended to an invariant of tied links, denoted by $\widetilde{\Upsilon}$, simply extending $ \pi_{\c}$ to $TB_n$ by mapping  $\eta_i$ to $E_i$. We denote this extension by $\widetilde{\pi}_{\c}$.

\begin{theorem}\label{T2}  The   function    $
\widetilde{\Upsilon} :\widetilde{\mathfrak{L}} \longrightarrow   \CC(\u,\v,\a,\sqrt{\c})$,  defines an invariant for tied links, where
$$
\widetilde{\Upsilon }(L) :=\left(\frac{1}{\a\sqrt{\c}}\right)^{n-1}(\rho \circ \widetilde{ \pi}_{\c})(\eta),
$$
$L$
 being the  closure of the  $n$--tied braid $\eta$.
\end{theorem}

 This theorem  will be proved together with   Theorem \ref{T3} of the next section.

\section{The invariant $\widetilde{\Upsilon}$ via skein relation}\label{TiedUpsilon}

This section  is two parts: the first one describe  $\widetilde{\Upsilon}$   by skein relation  and the second is devoted  to  analyze   a   specialization   $\widetilde\Omega$   of $\widetilde{\Upsilon}$.

In  the  sequel,   if  there  is  no risk  of  confusion,    we  indicate by $L$  both the  oriented   tied  link  and   its  diagram
  and we denote by  $L_+, L_-, L_\sim, L_{+,\sim}$    and $L_{-,\sim}$  the  diagrams  of  tied  links,  that  are  identical    outside   a  small  disc into  which enter two  strands, whereas  inside the disc the  two strands look  as  shown  in  Fig. 1.
\begin{figure}[H]
\centering
\includegraphics[scale=0.9]{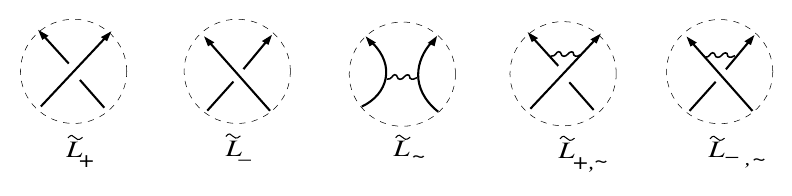}
\caption{}\label{ties3}
\end{figure}
The following  theorem  is the counterpart   of    \cite[Theorem 2.1]{aijuJKTR1}.

\begin{theorem} \label{T3}  The function $\widetilde{\Upsilon}$ is defined uniquely  by  the   following  three  rules: \begin{itemize}
\item[I] The  value of  $\widetilde{\Upsilon}$ is  equal to $1$  on  the  unknot.

\item[II] Let  $L$   be  a tied  link.  By  $  L \sqcup \bigcirc  $  we  denote  the   tied   link  consisting of $L$  and  the unknot,  unlinked to $L$.
Then
$$
\widetilde{\Upsilon}(L \sqcup  \bigcirc\,  )=  \frac {1}{\a \sqrt{\c} } \widetilde{\Upsilon}(L).
$$
\item[III] Skein  rule:
$$
 \frac{1}{ \sqrt{\c}}\widetilde{\Upsilon}( L_ +)-\sqrt{\c} \widetilde{\Upsilon} (L_-) =  \frac{\v-1}{\u} \widetilde{\Upsilon} (L_\sim) + \frac{1}{\sqrt{\c}} (1-\u^{-1} ) \widetilde{\Upsilon}(L_{+,\sim}).
$$

\end{itemize}
\end{theorem}
\begin{proof}  (of Theorems \ref{T2} and \ref{T3})
See the proof    done for the invariant  $\mathcal{F}$ in \cite[Theorem 2.1]{aijuJKTR1},   replacing  the variables $z$  and $w$ respectively  by $\a$ and  $\c$. The  definition of  $t$ must be replaced by  that of  $\b$ given by
 \begin{equation}\label{b} \b=\a (\u\c-1)/(1-\v), \end{equation}
 according to  (\ref{formulac}).
 All steps of the proofs are still holding for the new skein rules  involving the  new  parameter $\v$.

 At this point  we  have an  invariant for  tied links $\widetilde \Upsilon$,  uniquely defined  by  the  rules I--III. It  remains  to  prove  that  it  coincides with  that  obtained  via Jones  recipe:  the proof now proceeds exactly  as  that of \cite[Theorem 4.5]{aijuJKTR1}.  In this  way  we have proven  also Theorem \ref{T2}.

\end{proof}

\begin{remark}\rm  Rules I and  II  imply  that  the  value of  the  invariant  on  a collection of $n$ unlinked  circles  is  $ (\a\sqrt{\c})^{1-n}$.
\end{remark}

\begin{remark}\label{othersSkein}
\rm The following    skein rule IV   is obtained  from  rule III,  adding a  tie between  the  two  strands  inside  the  disc.  Rules Va and Vb are    equivalent  to  the  skein  rule  III, by using  rule IV.
 \begin{enumerate}
 \item[IV]
 \[ \frac{1}{\u \sqrt{\c}}\widetilde{\Upsilon}( L_ {+,\sim})- \sqrt{\c} \widetilde{\Upsilon} (L_{-,\sim}) =   \frac{\v-1}{\u} \widetilde{\Upsilon} (L_\sim). \]

\item[Va]
 \[ \frac{1}{ \sqrt{\c} }\widetilde{\Upsilon}( L_ {+})= \sqrt{\c} \left[\widetilde{\Upsilon}(L_-) +\left( \u-1 \right) \widetilde{\Upsilon}(L_{-,\sim})\right] + \left(\v-1\right) \widetilde{\Upsilon} (L_\sim). \]

\item[Vb]
\[   \sqrt{\c} \widetilde{\Upsilon}( L_ {-})= \frac{1}{ \sqrt{\c}} \left[\widetilde{\Upsilon}(L_+) +\frac{1-\u}{\u} \widetilde{\Upsilon}(L_{+,\sim})\right] + \frac{1-\v}{\u} \widetilde{\Upsilon} (L_\sim).\]
\end{enumerate}

\end{remark}

\begin{remark}\rm
 The  value of  the  invariant $\widetilde \Upsilon(\u,\v)$  on a tied link made by  $n$   unlinked  circles   all  tied, is obtained  by   rule  IV  (cf.  \cite[Remark 2.3]{aijuJKTR1}),   and  it is
\begin{equation}\label{ccc} \left( \frac {\u\c-1}{\sqrt{\c}(1-\v)}\right)^{n-1}= \left( \frac{ \b} { \a\sqrt{\c} } \right)^{n-1}  .\end{equation}
The last  equality comes from  (\ref{b}).
\end{remark}

 \begin{remark} \rm For  tied   links  in   $\widetilde{ \mathfrak{L}}^*$,  the invariant  $\widetilde \Upsilon$ is  uniquely defined  by  rules I and  IV. Observe that,  by  multiplying  skein rule IV by  $\sqrt{\u}$,  we  get  that $\widetilde \Upsilon$ coincides  with  the  Homflypt polynomial in the  variables $  t= \sqrt{\u \c}$   and
 $ x= (\v-1)/\sqrt{\c}$;  that is, if
 $L$ is the tied link  in $\widetilde{ \mathfrak{L}}^*$, associated to the classical links $L$, then
$
\widetilde \Upsilon(L) =\P( L).
$

 \end{remark}

 \begin{remark} \rm The  invariants of tied links  $\widetilde \Delta$  and  $\widetilde \Theta$  are, respectively,   the  specializations  $\widetilde \Upsilon(\u,\u)$  and  $\widetilde \Upsilon(1,\v)$.
 \end{remark}

\subsection{}

 We shall  denote by $\widetilde \Omega$ the specialization $\widetilde\Upsilon_{\u, \v
 }$ at  $\v=1$.

  We observe  firstly  that  if  $\v=1$  then  $\c=\u^{-1}$, so the  invariant $\widetilde\Omega
  $ takes in fact  values in  $\CC(\sqrt{\u},\a,\b)$.
 The next lemma describes $\widetilde \Omega$  by skein relations and is the key  to show its  main properties.

 \begin{lemma}\label{Omega} The  invariant  $\widetilde\Omega$  is uniquely  defined  by the following rules:
 \begin{itemize}\item[I]  $ \widetilde\Omega(\bigcirc)=1.$
 \item[II]
 $\widetilde\Omega(L  \sqcup  \bigcirc \,)=  \a^{-1}\sqrt{\u} \,  \widetilde\Omega(L).$
\item[III]  By  $  L \ \widetilde{\sqcup} \ \bigcirc \ $  we  denote  the   tied   link  consisting of the tied link  $L$  and  the unknot,  unlinked to $L$, but tied to one component of $L$.
Then
$$
\widetilde\Omega(L \,\widetilde{\sqcup} \ \bigcirc \,)=  \frac { \b \sqrt{\u}}{\a   } \,\widetilde\Omega(L).
$$
\item[IV] Skein  rule:
$$
 \sqrt{\u}\,\widetilde\Omega( L_ +)-\frac{1}{\sqrt{\u}}\,\widetilde\Omega(L_-) +   \sqrt{\u}  ( \u^{-1}-1 )\, \widetilde\Omega(L_{+,\sim})=0.
$$
  \end{itemize}
  \end{lemma}

 \begin{proof} By comparing   the rules of the lemma  with those  of Theorem \ref{T3}, we observe that: rule  I  coincides with  rule I for  $\widetilde \Upsilon$, rules II and IV  are obtained by  setting $\v=1$ in the  corresponding rules II  and III. Notice that, when the two components of the considered crossing are tied,  rule IV  becomes
 \begin{equation}\label{qp}
 \widetilde\Omega( L_{+,\sim})-\widetilde\Omega( L_{-,\sim})=0.
 \end{equation}
 Observe now   that  the necessity  of  rule III  for defining $\widetilde\Omega$, depends on the fact  that the skein rule IV
 does not involve the diagram $L_{\sim}$,  so  that  the value of $\widetilde\Omega$  on two unlinked circles  tied  together  cannot  be deduced, using skein rules; note this value is chosen by imposing that it coincide with the value obtained through  the Jones  recipe, and indeed  matches with (\ref{ccc}).     Rule III is in fact  the unique point  that makes the case $\v=1$  to be considered separately  from Theorem \ref{T3}.
 \end{proof}

  To present  the next result  we need to highlight  some facts and to introduce  some  notations.

   We start by recalling
   that the ties of a tied link define a partition of the set of components:   if there is a tie between two components, then  these components belong to the same class, see  \cite[Section 2.1]{aijuMathZ}.

   \begin{definition} \rm We  call    linking graph  of  a   link,  the $m$--graph  whose   vertices represent the $m$   components and  where two  vertices  are connected by  an edge   if  the corresponding components  have a  nonzero  linking  number.  Each edge is labeled  by  the  corresponding  linking number.
\end{definition}

   We  generalize the linking number  to  tied links.

\begin{definition} \rm We  call  class linking number  or c--linking number,  between  two  classes of components, the  sum of linking numbers of the components  of  one class  with  the components of the other class.
\end{definition}

  \begin{definition} \rm We  call     c--linking graph of  a tied link $\TL$,  the $k$--graph  whose   vertices represent the $k$ classes of the $\TL$ components  and  where  two  vertices  are connected by an edge  if  the corresponding classes  have a  nonzero  c--linking  number.  Each edge is labeled  by  the  corresponding c--linking number.
\end{definition}

\begin{example}\rm  The links  in Figure \ref{Fig2} have three  components: 1, 2 and 3,  and  two  classes,  $A= \{1,3\}$  and $B=\{2\}$.  All  crossings  have   positive sign.
The c--linking number between the  classes $A$ and $B$   is in both  cases equal to 2.   The corresponding c--linking  graph is shown at right.
\end{example}

   \begin{figure}[H]
\centering
\includegraphics[scale=0.8]{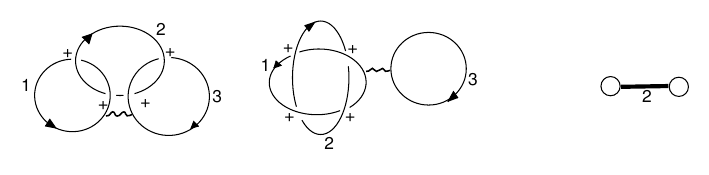}
\caption{}\label{Fig2}
\end{figure}

 \begin{remark}  \rm   For tied links in $\LLL$,  the c--linking graph coincides with the linking  graph.
 \end{remark}

 \begin{theorem} \label{T4} The invariant $\widetilde\Omega$  has the following properties:
\begin{enumerate}
\item The value of $\widetilde\Omega$ is equal to $1$  on knots.
\item $\widetilde\Omega$  takes  the same value  on  links  with  the same number of  components  all  tied together. The  value  depends only on the number of components  $m$, namely it is equal to  $(\b\sqrt{\u}/\a)^{m-1}$.
\item $\widetilde\Omega$  takes the same  value on tied links having   the same  number of  components and the same c--linking graph.
\end{enumerate}
\end{theorem}

\begin{proof}   Rule (\ref{qp})  implies that,  given  any knot diagram,
 $\widetilde\Omega$  takes the same  value on any other diagram obtained  by changing any  crossing   from positive to  negative or  viceversa.  Thus,  it takes the same value also on the  diagram  corresponding to the unknot:  by  rule I  this value is equal to 1.  This proves   claim (1).

Claim  (2) is a consequence of   rule (\ref{qp})  together with  rule III of Lemma \ref{Omega}.

Suppose the tied link $\widetilde L$ has  $m$ components,   partitioned into $k$ classes.  We order arbitrarily  the classes,  and inside each class, using rule (\ref{qp}),  we change the signs of some crossings in order to unlink the components   and  transform each component into the unknot.  Then we start from the first class $c_1$ and consider in their order all the other classes $c_i$ linked with it:  we mark  all  the undercrossing of $c_1$ with $c_i$   as {\it deciding} crossings.  Then we pass to  the   class $c_2$, we select  all classes $c_j$ linked with it and having   indices greater than 2, and   mark  the  undercrossings of  $c_2$   with $c_j$, so   increasing the list of  deciding  crossings. We proceed this way  till  the last  class.   At the end  we  have obtained an ordered sequence  of $q$ pairs of classes  characterized  by the corresponding   c--linking numbers. So,  we construct  a graph with $k$ vertices, and $q$ ordered  edges,  labeled with the  c--linking numbers.

Consider now the first pair of classes $(i,j)$ in the sequence. We apply  the  skein  rule  IV  of Lemma \ref{Omega}, to each one of the $n$ {\it deciding} crossings between  the components of this pair. These points  have signs $s_1,\ldots,s_n$. By  using  rule (\ref{qp}), rule IV becomes, respectively  for positive  and  negative crossings,    $$ \widetilde\Omega( L_ +)=\frac{1}{\u}\widetilde\Omega(L_-) +    ( 1-\frac{1}{\u}) \widetilde\Omega(L_{-,\sim} )\quad	 \text{and} \quad
    \widetilde\Omega( L_ -)= \u \widetilde\Omega(L_+) +     ( 1-\u   ) \widetilde\Omega(L_{+,\sim}).  $$
 So, consider the first deciding point with signs $s_1$. We have
 $$\widetilde\Omega(L_{s_1})=  \u^{-s_1}\widetilde\Omega( L_{-s_1})+( 1-\u^{-s_1}   ) \widetilde\Omega(L_{-s_1,\sim}).$$
The  two diagrams at the right member   are identical, but   in the second one there is a tie between the classes $i$ and $j$. We denote this diagram by  $L^{i\sim j}$; observe that in this diagram the classes  $i$ and $j$ merge in a sole class.

To calculate  the first term $\u^{-s_1}\widetilde\Omega( L_{-s_1})$, we pass to the second deciding point, so obtaining
 a first term  $\u^{-(s_1+s_2)}\widetilde\Omega( L_{-s_2}) $, and a second term  $\u^{-s_1}(1-\u^{-s_2} ) \OO(\TL^{i\sim j})$.  At the $n$-th  deciding point,  we obtain
$$\widetilde\Omega( L )= \u^{-(s_1+s_2+ \dots + s_n)} \widetilde\Omega ( L_{ -s_n})+
\sum_{i=1}^n \u^{-(s_0+\dots + s_{i-1})} (1-\u^{-s_i})  \OO(\TL^{i\sim j}),  $$
where $s_0=0$.  Now, $L_{-s_n}$ is the link obtained by $\TL$ by unlinking the classes $i$ and $j$, that we shall denote by $\TL^{i\parallel j}$.
By expanding the sum we obtain
$$\sum_{i=1}^n \u^{-(s_0+\dots + s_{i-1})} (1-\u^{-s_i})= 1- \u^{-(s_1+s_2+ \dots + s_n)} . $$
The  sum $s_1 + \dots + s_n$  is  the  sum of the signs  of all  undercrossings,  and therefore equals  the c--linking number of the two  classes, that we denote by $\ell(i,j)$. Therefore we get
\begin{equation}\label{skeinclass} \OO(\TL)= \u^{-\ell(i,j)} \OO(\TL^{i\parallel j}) +(1-\u^{-\ell(i,j)})\OO(\TL^{i\sim j}).  \end{equation}
\begin{figure}[H]
\centering
\includegraphics[scale=0.9]{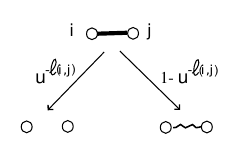}
\caption{Generalized skein rule}\label{Fig3}
\end{figure}

Observe now  that  Eq. (\ref{skeinclass}) is a {\it  generalized} skein relation, that is used to unlink  two  classes of components (or two components,  when the classes  contain a sole component), see Fig. \ref{Fig3}.   The  independence of the   calculation  by skein of  $\widetilde \Upsilon$   from the order of the deciding points,  implies  here the  independence of  the calculation of $\OO$  by  the generalized skein equation (\ref{skeinclass}) from the  order  of  the pairs of classes.

  However,  Eq. (\ref{skeinclass}) becomes $ \OO(\TL)= \OO(\TL^{i\parallel j})$  when the c--linking  number is  zero.  Therefore,  the invariant $\OO$  does  not  distinguish  two linked  classes with c--linking number zero  from  the same  classes   unlinked. So, we  obtain  the  c--linking  graph by deleting the   edges  labeled  by zero in the  graph  before  constructed. Observe  that, if  there  remain   $p$   edges   with non zero label, the  generalized  skein relation (\ref{skeinclass})  defines  a tree  terminating in    $2^p$  diagrams $\TL_j$, all  having the  classes  unlinked. These diagrams differ only for a  certain number of ties,  and each one of them  can be  represented  by a  graph obtained from the  c-linking  graph   where  each  edge  is either  deleted or substituted by a  tie.
 The value of $\OO(\TL)$ is  then the sum
 \begin{equation}\label{OO} \sum_{j=1}^{2^p} \alpha_j \OO(\TL_j).\end{equation}
 Notice that  each vertex of the tree is labeled by a pair $(x,y)$ of classes, that is, the classes that are unlinked by the skein rule at that vertex,  To calculate the coefficient $\alpha_j$,    consider  all the $p$  vertices of    the   path in the skein tree, going from $\TL_j$ to  $\TL$.  For  each one of these vertices,  say with label $(x,y)$,   choose the factor   $\u^{-\ell(x,y)}$  if it is reached from left, otherwise the factor  $(1-\u^{-\ell(x,y)})$.  The coefficient $\alpha_j$ is the product of   these  $p$ factors.

 The value   $\OO(\TL_j)$    depends only on the  number $m$ of components,   and on the  number of classes  $h$, $h\le k$, of $\TL_j$;  indeed, by    rules  II and III of Lemma \ref{Omega}  we have:
\begin{equation}\label{mc}\widetilde\Omega(\TL_j)= \left(\frac{\sqrt \u}{\a}\right)^{m-1} \b^{m-h}.  \end{equation}

  To calculate   $h$ for the diagram $\TL_j$, we start from  the c--linking graph of $\TL$, and use again  the $p$ vertices of the considered path in the   skein  tree:
    if  the path reaches a vertex labeled $(x,y)$ from left,  then the edge $(x,y)$    is eliminated  from the graph, otherwise  the edge is  substituted by a  tie. The  number of connected components of the graph so obtained,  having ties as edges, is  the resulting number $h$ of  classes, e.g. see Fig. \ref{Fig5}.

  To  conclude the proof,  it is now sufficient to observe that the  calculation of $\OO(\TL)$   depends only  on the  c--linking graph  and on the total number of components of $\TL$.

\end{proof}

\begin{corollary} Let  $\TL$  be a tied link  with $m$ components and $k$ classes. Let $r$ be the  exponent of $\a$   and $s$  the  minimal exponent of $\b$ in  $\OO(\TL)$. Then $m=1-r$ and $k=1-r-s$.
\end{corollary}

\begin{proof}  It follows from Eq. (\ref{OO})  and (\ref{mc}),  noting that  the coefficients $\alpha_i$ depend only on the variable $\u$.
\end{proof}

  We shall denote   by $\Omega$ the specialization of $\Upsilon$ at $\v=1$, that is,    $\Omega$ is   the restriction of $\widetilde{\Omega}$ to $  \mathfrak{L}$. We have the following results for classical links.

\begin{corollary} The invariant $\Omega$  has the following properties:
  $\Omega$  takes  the same value    on  links having   the same   linking graph. If $L$ has $m$ components, the exponent of $\a$ in $\Omega$ is  $1-m$ and there is a term in $\Omega$  non containing $\b$.
\end{corollary}

\begin{proof}  It follows from Theorem 4 and Corollary 1.
\end{proof}

\begin{example} \rm Consider  the   link  $ L$ in Figure \ref{Fig4}. Here $m=3$, $c=3$ and    $p=3$.
All linking numbers  $\ell(i,j)$  are equal to 1.
\begin{figure}[H]
\centering
\includegraphics[scale=0.9]{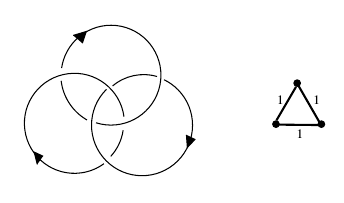}
\caption{A link and its linking graph}\label{Fig4}
\end{figure}
The value $\Omega(L)$ is  obtained   by  adding   the   value of $\Omega$ on  the  $2^3$ graphs shown in  Figure \ref{Fig5}, where they are subdivided   in four groups, according to  the value of $\Omega$, i.e., to the number of classes, indicated  at bottom.
The coefficients, here written for  the four  groups, are:
$$    \u^{-3} , \quad     \u^{-2}(1-\u^{-1}),  \quad    \u^{-1}(1-\u^{-1})^2 \quad    \text{and} \quad  (1-\u^{-1})^3,  $$
whereas the corresponding values of $\Omega$  are
$$     \u/ \a^2,   \quad   \b \u/ \a^2 ,   \quad  \b^2 \u/ \a^2  \quad   \text{and} \quad  \b^2   \u/ \a^2. $$
  Then,
$\Omega(L)= \u\a^{-2}( \u^{-3} + 3 \b(\u^{-2}(1-\u^{-1}) )+ 3\b^2 (\u^{-1}(1-\u^{-1})^2 )+ \b^2(1-\u^{-1})^3 )$,  so
$$
\Omega(L)=  \a^{-2}\u^{-2} (1+3\b \u-3\b-3\b^2\u+2\b^2+\b^2\u^3).
$$
\begin{figure}[H]
\centering
\includegraphics[scale=0.9]{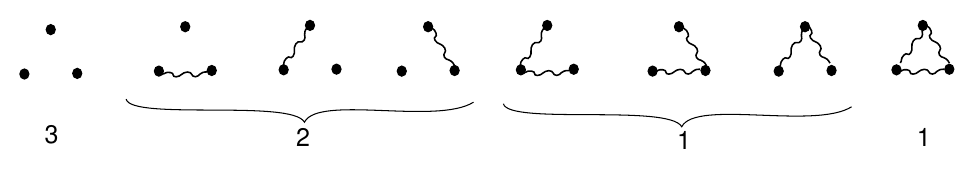}
\caption{The  eight graph obtained from the  c--linking  graph  of  Fig. \ref{Fig4}}\label{Fig5}
\end{figure}
Finally, observe that   $r=-2$ and $s=0$; indeed,  $L$  has 3 components and 3 classes.
\end{example}

 \section{Results of calculations}\label{Computations}

 Here the  notations   of the links with ten or eleven  crossings    are taken from \cite{chli}.

 The  following table  shows eight pairs of non isotopic links with three  components,  distinguished by $\Upsilon(\u,\v)$, but non distinguished by  the Homflypt polynomial.  A star indicates when they are   distinguished   also  by a specialization of $\Upsilon(\u,\v)$.


 $$  \begin{matrix}
 \hline     Link & l. graph& Link & l. graph &\Upsilon(\u,\v)\ & \Upsilon(1,\v) &   \Upsilon(\u,\u) &      \Upsilon(\u,1)\\ \hline
 L11n358  \{0,1\} &\includegraphics[scale=0.1]{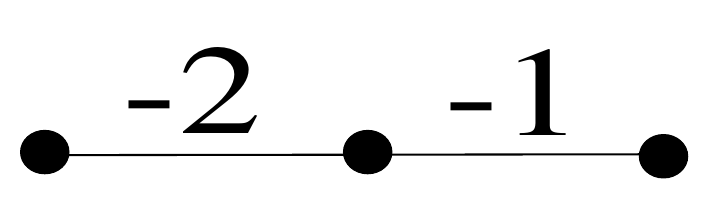} & L11n418  \{0,0\}& \includegraphics[scale=0.1]{G1ab.pdf}& \star & \star  &      &       \\ \hline
 L11n358  \{1,1\} &\includegraphics[scale=0.1]{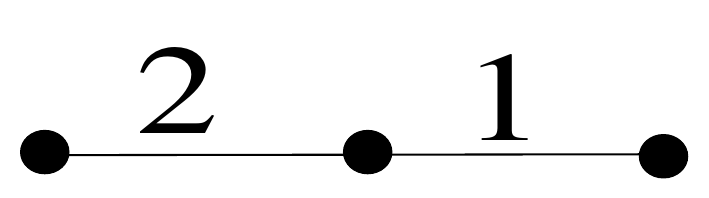} & L11n418  \{1,0\}& \includegraphics[scale=0.1]{G2ab.pdf}& \star &   & \star     &       \\ \hline
  L11n356  \{1,0\}&\includegraphics[scale=0.1]{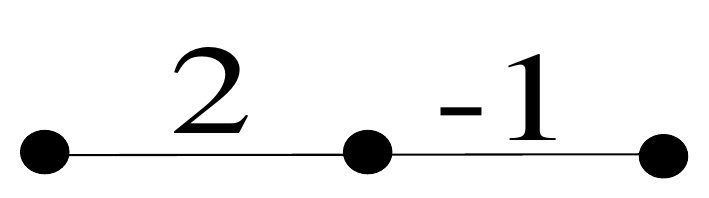}  & L11n434  \{0,0\}&\includegraphics[scale=0.1]{G3ab.pdf} &  \star & &  \star      &    \\ \hline
  L11n325  \{1,1\}& \includegraphics[scale=0.1]{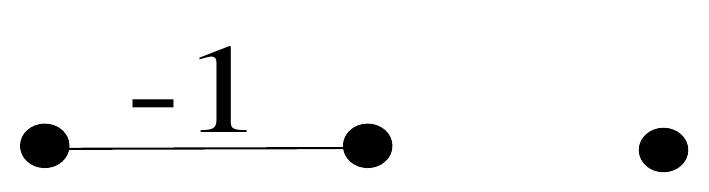} & L11n424  \{0,0\}& \includegraphics[scale=0.1]{G4ab.pdf}&  \star & \star   &  \star     &   \\ \hline
  L10n79  \{1,1\} & \includegraphics[scale=0.1]{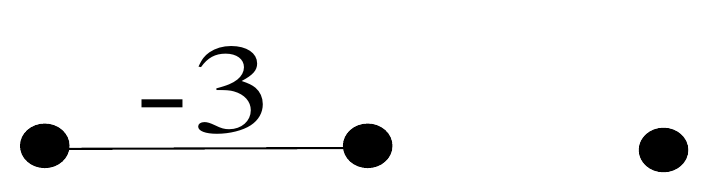} & L10n95  \{1,0\} & \includegraphics[scale=0.1]{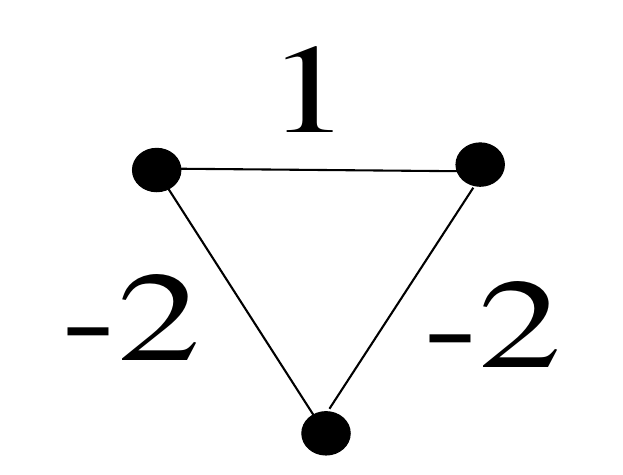}&  \star   &  \star  &  \star      &  \star \\ \hline
   L11a404  \{1,1\} &\includegraphics[scale=0.1]{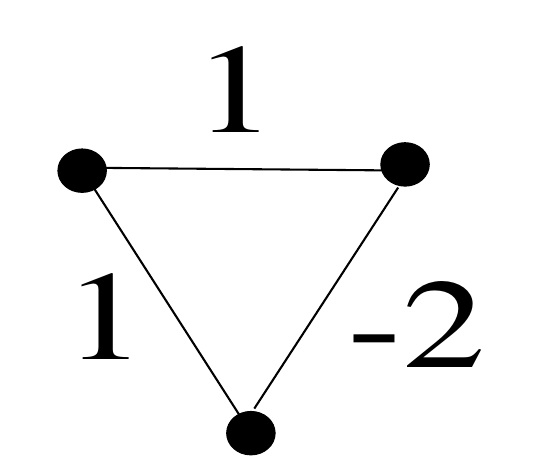} &L11a428  \{1,0\} & \includegraphics[scale=0.1]{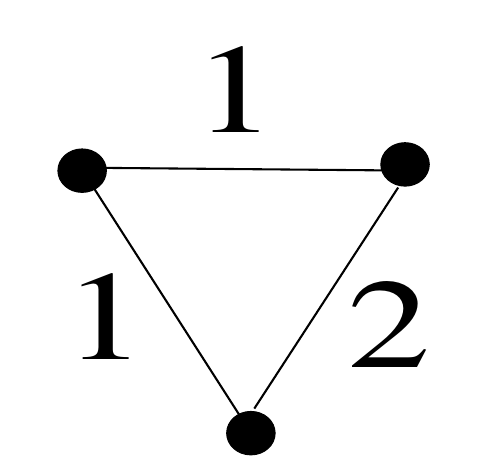}&   \star  &  \star  &  \star   & \star \\ \hline
  L11a467  \{0,1\}&\includegraphics[scale=0.1]{G1ab.pdf}  & L11a527 \{0,0\} &  \includegraphics[scale=0.1]{G1ab.pdf} & \star & \star  & &   \\  \hline
  L10n76\{1,1\}   & \includegraphics[scale=0.1]{G5a.pdf}&  L11n425\{1,0\} & \includegraphics[scale=0.1]{G5b.pdf} &\star   & \star   &  \star  & \star  \\  \hline
   \end{matrix}$$

   Observe  that,  among the  eight pairs  distinguished by  $ \Upsilon(\u,\v)$,  six  are distinguished by $ \Upsilon(\u,\u)$,   six by  $ \Upsilon(1,\v)$;  the pair  distinguished  by  both $\Upsilon(\u,\u)$ and $ \Upsilon(1,\v)$ are four, three  of which   are  distinguished  also  by  $\Upsilon(\u,1)$. We don't know whether it is  necessary, for being  distinguished by  $\Upsilon(\u,1)$, to  be  distinguished by  all other  specializations.

\section{Final  Remarks}\label{Digression}
\subsection{}
Recently, Jacon and Poulain d'Andecy have constructed an explicit isomorphism between the Yokonuma--Hecke algebra and a direct sum of matrix algebras over tensor products of Iwahori--Hecke algebras, also they  have classified  the Markov traces on the Yokonuma--Hecke algebra,  see \cite[Theorems  3.1, 5.3]{JaPoMathZ}. In the same paper they defined an invariant of three parameters and have shown that the invariants $\Theta$ and $\Delta$ can be obtained from it, see \cite[Subsection 6.5]{JaPoMathZ}.  In \cite[Appendix]{chjukalaIMRN}  Lickorish found a formula  for $\Theta$ which allows to compute  $\Theta(L)$ through the linking numbers and the Homflypt polynomials of the sublinks of the oriented link $L$, cf.  \cite{PoWaProc}.
On the other hand, in \cite[Theorem 14]{esryJPAA} Espinoza and Ryom--Hansen proved that the bt--algebra can be considered as a subalgebra of the Yokonuma--Hecke algebra; cf. \cite{jaPoJLMS}.  This result together with the results  of Jacon and Poulain d'Andecy and  the formula for $\Theta$ induce to think that some of the  results of this paper, at level of classical links, could be recovered by a combination of the results mentioned before. However, this combination do not imply  the results proved here for tied links.

An open problem yet is to know how strong is the  four variable  invariant  $\Upsilon$ (respectively, $\widetilde \Upsilon$)   with respect to $\Theta$ and $\Delta$  (respectively,  $\widetilde{\Theta}$ and $\widetilde{\Delta}$).

\subsection{}
Denote by  $\mathrm{H}_n(\u)$ the Hecke algebra, that is,  the $\CC(\u)$--algebra generated by $h_1, \ldots , h_{n-1}$ subject to the braid relations of type $A$,  together with  the quadratic relation
$$
h_i^2 = \u + (\u-1)h_i,\quad\text{for all $i$.}
$$
Now, there exits  another presentation used to describe the Hecke algebra, which is obtained by  rescaling  $h_i$ by $\sqrt{\u}^{\,-1}$; more precisely, taking $f_i:= \sqrt{\u}^{\,-1}h_i$.  In this case  the  $f_i$'s satisfy the braid relations and the quadratic relation
$$
f_i^2 = 1 + (\sqrt{\u} - \sqrt{\u}^{\,-1}) f_i.
$$
Denote by $\mathrm{H}_n(\sqrt{\u})$ the presentation of the Hecke algebra through the $f_i$'s.   The construction of the Homflypt  polynomial can be made indistinctly from any of the  above presentations for the Hecke algebra.

The bt--algebra can be regarded as a generalization of the Hecke algebra, in the sense that, by taking $E_i=1$ in the presentation of the bt--algebra,  we get the Hecke algebra; indeed, under $E_i=1$ the presentations, respectively,  of  $\E_n(\u)$ and  $\E_n(\sqrt{\u})$  becomes
 $\mathrm{H}_n(\u)$ and    $\mathrm{H}_n(\sqrt{\u})$. Now  we recall that, as we noted in  observation 5 of  Subsection \ref{NotationFact}, these two presentations  of the bt--algebra  yield  different invariants. The authors don't know other situations where different presentations of the same algebra produce  different invariants.

\subsection{}  Also the Hecke algebra with two parameters can be   considered; that is, by taking  two commuting parameters $\u_1$ and $\u_2$, and  imposing that   that the generators $h_i$'s satisfy   $h_i^2 = \u _1+ \u_2 h_i$, for all $i$; however, the   Hecke algebras with  one  and two    parameters are isomorphic, see \cite[Subsection 4.2]{katu}; hence, from the algebraic point of view  these algebras are the same. Now,  regarding the   behavior of the Hecke algebra with two parameters   $\mathrm{H}_n(\u_1, \u_2)$,  in the construction of polynomial invariants,  we have that, after suitable rescaling,   $\mathrm{H}_n(\u_1, \u_1)$ becomes  of the type $\mathrm{H}_n(\sqrt{\u})$  and $\mathrm{H}_n(\u_1, 0)$ becomes    the group algebra of the symmetric group. For
$\mathrm{H}_n(0, \u_2)$, we   obtain the so--called $0$--Hecke algebra.

We examine  now the bt--algebra with   one more parameter. Taking  $\u_0,\u_1,\u_2$ and $\u_3$ commuting variables, it is natural to keep generators $R_i$'s instead the $T_i$'s, satisfying $R_i^2= \u_0 + \u_1E_i + \u_2E_iR_i +\u_3 R_i$, for all $i$; notice that a simple rescaling shows that we can   take $\u_0=1$. Now, we need that these $R_i$'s, together with the $E_i$'s, satisfy all defining relations of the bt--algebra with the only exception of relation (\ref{quadraticT}); it is straightforward  to see that these defining relations hold  if and only if we take $\u_3=0$. This  is the motivation for defining  the  bt--algebra   $\En(\u,\v)$  with two parameters in this paper.  Observe that we have a homomorphism from  $\En(\u,\v)$ onto $\mathrm{H}_n(\u, \v-1)$,  defined by sending $E_i$ to $1$ and $R_i$ to $h_i$; so, the $0$--Hecke algebra is the homomorphic image of $\E_n(0,\v)$.

  A  natural  question is why do not study a further  generalization  of  the bt--algebra,  say with three   variables  $\u$, $\v$ and $\w$   commuting among them, and  with  quadratic  relation $$
R_i^2 = 1 +(\u-1)E_i + (\v-1)E_iR_i + (\w-1)R_i,\quad \text{for all $i$}.
$$  In fact,  we  have proved  that this relation  is not compatible  with  some   monomial  relations of  the  bt--algebra.

\end{document}